\documentclass[12pt, reqno]{amsart}
\usepackage{amssymb,latexsym,amsmath,amscd,amsthm,graphicx, color}
\usepackage[all]{xy}
\usepackage{pgf,tikz}
\usepackage{mathrsfs}
\usetikzlibrary{arrows}
\usepackage[left=0.6 in, top=0.6 in, right=0.6 in, bottom=0.5 in]{geometry}

\usepackage[linkcolor=blue, urlcolor=blue, citecolor=blue,
colorlinks, bookmarks]{hyperref}

\definecolor{uuuuuu}{rgb}{0.26666666666666666,0.26666666666666666,0.26666666666666666}
\definecolor{xdxdff}{rgb}{0.49019607843137253,0.49019607843137253,1.}
\definecolor{ffqqqq}{rgb}{1.,0.,0.}
\definecolor{ffqqqq}{rgb}{1.,0.,0.}
\definecolor{ffxfqq}{rgb}{1.,0.4980392156862745,0.}

\definecolor{uuuuuu}{rgb}{0.26666666666666666,0.26666666666666666,0.26666666666666666}
\definecolor{qqwuqq}{rgb}{0.,0.39215686274509803,0.}
\definecolor{zzttqq}{rgb}{0.6,0.2,0.}
\definecolor{xdxdff}{rgb}{0.49019607843137253,0.49019607843137253,1.}
\definecolor{qqqqff}{rgb}{0.,0.,1.}
\definecolor{cqcqcq}{rgb}{0.7529411764705882,0.7529411764705882,0.7529411764705882}
\definecolor{sqsqsq}{rgb}{0.12549019607843137,0.12549019607843137,0.12549019607843137}

\theoremstyle{plain}

\newtheorem{theorem}[subsection]{Theorem}

\newtheorem{corollary}[subsection]{Corollary}
\newtheorem{lemma}[subsection]{Lemma}
\newtheorem{prop}[subsection]{Proposition}

\theoremstyle{definition}
\newtheorem{defi1}[subsection]{Definition}

\newtheorem{remark}[subsection]{Remark}

\newcommand{\uu}{\cup}
\newcommand{\UU}{\bigcup}

\newcommand{\ci}{\subseteq}
\newcommand{\sci}{\subset}

\newcommand{\es}{\emptyset}
\newcommand{\set}[1]{\{#1\}}

\newcommand{\ga}{\alpha}
\newcommand{\gb}{\beta}

\newcommand{\gd}{\delta}
\renewcommand{\gg}{\gamma}

\newcommand{\gk}{\kappa}

\newcommand{\gs}{\sigma}
\newcommand{\gt}{\tau}

\newcommand{\tbf}{\textbf}
\newcommand{\tit}{\textit}

\newcommand{\D}[1]{\mathbb{#1}}

\newcommand{\te}{\text}

\begin{document}
\title{Constrained quantization for the Cantor distribution with a family of constraints}

\author{$^1$Megha Pandey}
 \author{$^2$Mrinal K. Roychowdhury}

\address{$^{1}$Department of Mathematical Sciences \\
Indian Institute of Technology (Banaras Hindu University)\\
Varanasi, 221005, India.}
\address{$^{2}$School of Mathematical and Statistical Sciences\\
University of Texas Rio Grande Valley\\
1201 West University Drive\\
Edinburg, TX 78539-2999, USA.}

\email{$^1$meghapandey1071996@gmail.com, $^2$mrinal.roychowdhury@utrgv.edu}

\subjclass[2010]{Primary 28A80; Secondary 94A34, 60Exx.}
\keywords{Cantor distribution, constrained quantization error, optimal sets of $n$-points, constrained quantization dimension, constrained quantization coefficient}

\date{}
\maketitle

\pagestyle{myheadings}\markboth{M. Pandey and M.K. Roychowdhury}{Constrained quantization for the Cantor distribution with a family of constraints}
\begin{abstract}
In this paper, for a given family of constraints and the classical Cantor distribution we determine the constrained optimal sets of $n$-points, $n$th constrained quantization errors for all positive integers $n$. We also calculate the constrained quantization dimension and the constrained quantization coefficient, and see that the constrained quantization dimension $D(P)$ exists as a finite positive number, but the $D(P)$-dimensional constrained quantization coefficient does not exist. 

\end{abstract}

\section{Introduction}

Constrained quantization for a Borel probability measure refers to the idea of estimating a given probability by a discrete probability with a finite number of supporting points lying on a specific set. The specific set is known as the constraint of the constrained quantization. A quantization without a constraint is known as an unconstrained quantization, which traditionally in the literature is known as quantization. Constrained quantization has recently been introduced by Pandey and Roychowdhury (see \cite{PR1, PR2}). Recently, they have also introduced the conditional quantization in both constrained and unconstrained quantization (see \cite{PR4}). For some follow up papers in the direction of constrained quantization, one can see \cite{PR3, BCDRV}). With the introduction of constrained quantization, quantization now has two classifications: constrained quantization and unconstrained quantization. 
For unconstrained quantization and its applications one can see \cite{DFG, DR, GG, GL, GL1, GL2, GL3, GN, KNZ, P, P1, R1, R2, R3, Z1, Z2}. constrained quantization has many significant real world applications. Constrained quantization is greatly useful in radiation therapy of cancer treatment. In radiation therapy, to make sure that the radiation does not directly hit the region of good cells one can use the constrained quantization technique. Constrained quantization is also useful in sending and getting signals from a certain region with minimum distortion using a fixed number of towers installed in a different region. 
 
\begin{defi1}\label{Vr}
Let $P$ be a Borel probability measure on $\D R^2$ equipped with a metric $d$ induced by the Euclidean norm $\|\cdot\|$ on $\D R^2$. 
Let $\set{S_j\ci \D R^2: j\in \D N}$ be a family of closed sets with $S_1$ nonempty. Then, for $n\in \mathbb{N}$, the \tit {$n$th constrained quantization
error} for $P$ with respect to the family of constraints $\set{S_j\ci \D R^2: j\in \D N}$, is defined as
\begin{equation} \label{EqVr}
V_{n}:=V_{n}(P)=\inf \Big\{\int \mathop{\min}\limits_{a\in\ga} d(x, a)^2 dP(x) : \ga \ci \UU_{j=1}^nS_j, ~ 1\leq  \text{card}(\ga) \leq n \Big\},
\end{equation}
where $\te{card}(A)$ represents the cardinality of the set $A$.
\end{defi1}
The number 
\begin{equation*}
  V(P; \ga):= \int \mathop{\min}\limits_{a\in\ga} d(x, a)^2 dP(x)
\end{equation*}
is called the distortion error for $P$ with respect to a set $\ga \ci \D R^2$. Write $V_{\infty}(P):=\mathop{\lim}\limits_{n\to \infty} V_{n}(P)$.  
Then, the number $D(P)$ defined by 
 \[D(P):=\mathop{\lim}\limits_{n\to \infty} \frac{2\log n}{-\log (V_{n}(P)-V_{\infty}(P))},  \]
if it exists, is called the \tit{constrained quantization dimension} of $P$. 
For any $\gk>0$, the  number  
\begin{equation} \label{eq00100} \lim_{n\to \infty} n^{\frac 2 \gk}  (V_{n}(P)-V_{\infty}(P)),\end{equation} if it exists, is called the \tit{$\gk$-dimensional constrained quantization coefficient} for $P$.  
 Let us take the family $\set{S_j : j\in \D N}$ of constraints, that occurs in  \eqref{EqVr} as follows:
\begin{equation} \label{eq000} S_j=\set{(x, y) : -\frac 1 j\leq x\leq 1 \te{ and } y=x+\frac 1 j } 
\end{equation}
for all $j\in \D N$. 
Let $T_1, T_2 : \mathbb R \to \mathbb R$ be two contractive similarity mappings such that
$T_1(x)=\frac 13 x$ and $ T_2 (x)=\frac 13 x +\frac 23$. Then, there exists a unique Borel probability measure $P$
on $\D R$ such that
$P=\frac 12 P\circ T_1^{-1}+\frac 12 P\circ T_2^{-1}$, where $P\circ T_i^{-1}$ denotes the image measure of $P$ with respect to
$S_i$ for $i=1, \,2$ (see \cite{H}). If $k\in \D N$, and $\gs:=\gs_1\gs_2 \cdots \gs_k \in \{ 1, 2\}^k$, then we call $\gs$ a word of length $k$ over the alphabet $I:=\set{1, 2}$, and denote it by
$|\gs|:=k$. By $I^\ast$, we denote the set of all words including the empty word $\es$. Notice that the empty word has length zero. For any word $\gs:=\gs_1\gs_2 \cdots \gs_k \in I^\ast$,  we write \[T_\gs:=T_{\gs_1}\circ \cdots \circ T_{\gs_k} \text{ and } J_\gs:=T_\gs([0, 1]).\]
Then, the set $C:=\bigcap_{k\in \mathbb N} \bigcup_{\gs \in \{1, 2\}^k} J_\gs$ is known as the \textit{Cantor set} generated by the
two mappings $T_1$ and $T_2$, and equals the support of the probability measure $P$, where $P$ can be written as
 \[P=\sum_{\gs\in \set{1, 2}^k} \frac 1 {2^k}  P\circ T_\gs^{-1}.\]
For this probability measure $P$, Graf and Luschgy determined the optimal sets of $n$-means and the $n$th quantization errors for all $n\in \D N$ (see \cite{GL2}) in unconstrained scenario. They also showed that the unconstrained quantization dimension of the measure $P$ exists and equals $\frac{\log 2}{\log 3}$, which is the Hausdorff dimension of the Cantor set $C$, and the unconstrained quantization coefficient does not exist. In fact, in \cite{GL2}, they showed that the lower and the upper quantization coefficients exist as finite positive numbers. 

In this paper, with respect to the family of constraints $\set{S_j : j\in \D N}$ for the Cantor distribution $P$ we determine the constrained optimal sets of $n$-points and the $n$th constrained quantization errors for all positive integers $n$. We further show that the constrained quantization dimension of the Cantor distribution $P$ exists and equals two. Moreover, the value of the constrained quantization coefficient comes as infinity. From the work in this paper, we see that the constrained quantization dimension and the constrained quantization coefficient for the classical Cantor distribution depend on the family of constraints. 
 
\section{Preliminaries}
In this section, we give some basic notations and definitions which we have used throughout this paper. 
As defined in the previous section, let $I:=\{1, 2\} $ be an alphabet. For any two words $\gs:=\gs_1\gs_2\cdots \gs_k$ and
$\tau:=\tau_1\tau_2\cdots \tau_\ell$ in $I^*$, by
$\gs\tau:=\gs_1\cdots \gs_k\tau_1\cdots \tau_\ell$, we mean the word obtained from the concatenation of the two words $\gs$ and $\tau$. For $\gs, \gt\in I^\ast$, $\gs$ is called {\it an extension of} $\gt$ if $\gs=\gt x$ for some word $x\in I^\ast$. The mappings $T_i:\D R \to \D R,\ 1\leq i \leq 2, $ such that $T_1(x)=\frac 13x$ and $T_2(x)=\frac 13x+\frac 23$ are the generating maps of the Cantor set $C $, which is the support of the probability measure $P$ on $\D R$ given by $P=\frac 12 P\circ T_1^{-1}+\frac 12 P\circ T_2^{-1}$. For $\gs:=\gs_1\gs_2 \cdots\gs_k \in I^k$, write 
$J_\gs=T_\gs [0, 1]$, where $T_\gs:=T_{\gs_1}\circ T_{\gs_2}\circ\cdots \circ T_{\gs_k}$ is a composition mapping. Notice that $J:=J_\es=T_\es[0, 1]=[0,1]$. 
Then, for any $k\in \D N$, as mentioned before, we have 
\[C=\bigcap_{k\in \mathbb N} \bigcup_{\gs \in I^k} J_\gs \te{ and  } P=\sum_{\gs \in I^k}\frac 1 {2^k} P\circ T_\gs^{-1}.\]
The elements of the set $\{J_\gs : \gs \in I^k \}$ are the $2^k$ intervals in the $k$th level in the construction of the Cantor set $C$, and are known as the {\it basic intervals at the $k$th level.}  The intervals $J_{\gs 1}$, $J_{\gs 2}$, into which $J_\gs$ is split up at the $(k+1)$th level are called the {\it children of $J_\gs$.}

With respect to a finite set $\ga \sci \D R^2$, by the \tit{Voronoi region} of an element $a\in \ga$, it is meant the set of all elements in $\D R^2$ which are nearest to $a$ among all the elements in $\ga$, and is denoted by $M(a|\ga)$.
Let 
\begin{equation} \label{Me00} \rho: \D R \times \D R^2 \to [0, \infty) \te{ such that } \rho(x, (a, b))=(x-a)^2 +b^2,\end{equation} 
where $x\in \D R$ and $(a, b) \in \D R^2$, which defines a nonnegative real-valued function on $\D R \times \D R^2$. Notice that the function $\rho$ gives the squared Euclidean distance between an element in $\D R$ and an element in $\D R^2$. Let $\pi: \D R^2 \to \D R$ such that $\pi(a, b)=a$ for any $(a, b) \in \D R^2$ denote the projection mapping. For a random variable $X$ with distribution $P$, let $E(X)$ represent the expected value, and $V:=V(X)$ represent the variance of $X$. 

The following lemmas are well-known (see \cite{GL2}).  
\begin{lemma} \label{lemma1}
Let $f : \mathbb R \to \mathbb R^+$ be Borel measurable and $k\in \mathbb N$. Then
\[\int f dP=\sum_{\gs \in \{1, 2\}^k} p_\gs \int f \circ S_\gs dP.\]
\end{lemma}

 \begin{lemma} \label{lemma2}
Let $ X$ be a random variable with probability distribution $P.$  Then, $E(X)= \frac 12  \te{ and } V:=V(X)=E\|X-\frac 1 2 \|^2=E(X-\frac 12)^2=\frac 1 8.$ Moreover, for any $x_0\in \D R$, we have 
\[\int (x-x_0)^2 dP(x)=V(X)+(x-\frac 12)^2.\]
\end{lemma}

\begin{remark}
For words $\gb, \gg, \cdots, \gd$ in $I^\ast$, by $a(\gb, \gg, \cdots, \gd)$ we mean the conditional expectation of the random vector $ X$ given $J_\gb\uu J_\gg \uu\cdots \uu J_\gd,$ i.e.,
\begin{equation*} \label{eq45} a(\gb, \gg, \cdots, \gd)=E(X : X\in J_\gb \uu J_\gg \uu \cdots \uu J_\gd)=\frac{1}{P(J_\gb\uu \cdots \uu J_\gd)}\int_{J_\gb\uu \cdots \uu J_\gd} x \, dP.
\end{equation*}
Recall Lemma \ref{lemma1}. For each $\gs \in I^\ast$, since $T_\gs$ is a similarity mapping, we have
\begin{align*}
a(\gs)&=E(X : X \in J_\gs) =\frac{1}{P(J_\gs)} \int_{J_\gs} x \,dP=\int_{J_\gs} x  d(P\circ T_\gs^{-1})=\int T_\gs (x) \, dP\\
&=E(T_\gs(X))=T_\gs(E(X))=T_\gs(\frac 12).
\end{align*}
\end{remark}
In this paper, we investigate the constrained quantization for the family of constraints given by 
 \begin{equation} \label{eq000} S_j=\set{(x, y) : -\frac 1 j\leq x\leq 1 \te{ and } y=x+\frac 1 j } \te{ for all } j\in\D N, 
\end{equation}
i.e., the constraints $S_j$ are the line segments joining the points $(-\frac 1 j, 0)$ and $(1, 1+\frac 1 j)$ which are parallel to the line $y=x$.
The perpendicular on a constraint $S_j$ passing through a point $(x, x+\frac 1 j)\in S_j$ intersects the real line at the point $2x+\frac 1 j$ if $-\frac 1 j\leq x\leq 1$; and it intersects $J$ if 
$0\leq 2x+\frac 1 j\leq 1$, i.e., if 
\begin{equation} \label{eq0000} -\frac 1 {2j}\leq x\leq\frac 1 2-\frac 1 {2j}.
\end{equation} 
 Thus, for all $j\in \D N$, there exists a one-one correspondence between the elements $(x, x+\frac 1 j)$ on $S_j$ and the elements $2x+\frac 1 j$ on the real line if $-\frac 1 j\leq x\leq 1$. Thus, for all $j\in \D N$, there exist bijective mappings $U_j$ such that
 \begin{equation} \label{eq0001} 
U_j(x, x+\frac 1j)=2x+\frac 1 j \te{ and } U_j^{-1}(x)=\Big(\frac 1 2 (x-\frac 1 j), \frac 1 2 (x-\frac 1 j)+\frac 1 j\Big), \end{equation}
where $-\frac 1 j\leq x\leq 1$.

The following lemma plays an important role in the paper.

\begin{lemma} \label{lemma0}
Let $\ga_n\ci \mathop{\uu}\limits_{j=1}^n S_j$ be a constrained optimal set of $n$-points for $P$ such that
\[\ga_n:=\set{(a_j, b_j) : 1\leq j\leq n},\]
where $a_1<a_2<a_3<\cdots<a_n$ and $\pi$ be the projection mapping. Then, $(a_j, b_j)=U_n^{-1}(E(X :  X\in \pi(M((a_j, b_j)|\ga_n))))$,
where $M((a_j, b_j)|\ga_n)$ are the Voronoi regions of the elements $(a_j, b_j)$ with respect to the set $\ga_n$ for $1\leq j\leq n$.
 \end{lemma}

\begin{proof}
Let $\ga_n:=\set{(a_j, b_j): 1\leq j\leq n}$, as given in the statement of the lemma, be a constrained optimal set of $n$-points. Take any $(a_q, b_q)\in \ga_n$. Since $\ga_n\ci \mathop{\uu}\limits_{j=1}^n S_j$, we can assume that $(a_q, b_q) \in S_t$ for some $1\leq t\leq n$. Since the Voronoi region of $(a_q, b_q)$, i.e., $M((a_q, b_q)|\ga_n)$ has positive probability, $M((a_q, b_q)|\ga_n)$ contains some basic intervals from $J$ that generates the Cantor set $C$. Let $J_{\gs^{(j)}}$, where $1\leq j\leq k$ for some positive integer $k$, be all the basic intervals that are contained in $M((a_q, b_q)|\ga_n)$. Now, the distortion error contributed by $(a_q, b_q)$ in its Voronoi region $M((a_q, b_q)|\ga_n)$ is given by
\begin{align*}
&\int_{\pi(M((a_q, b_q)|\ga_n))}\rho(x, (a_q, b_q)) \,dP\\&=\sum_{j=1}^k \frac 1{2^{\ell(\gs^{(j)})}} \int_{J_{\gs^{(j)}}}\rho(x, (a_q, b_q)) \,d(P\circ T_{\gs^{(j)}}^{-1})\\
&=\sum_{j=1}^k\frac 1{2^{\ell(\gs^{(j)})}} \frac 1{9^{\ell(\gs^{(j)})}}V+\sum_{j=1}^k\frac 1{2^{\ell(\gs^{(j)})}}\rho(T_{\gs^{(j)}}(\frac 12), (a_q, a_q+\frac 1 t))\\
&=\sum_{j=1}^k\frac 1{2^{\ell(\gs^{(j)})}} \frac 1{9^{\ell(\gs^{(j)})}}V+\sum_{j=1}^k\frac 1{2^{\ell(\gs^{(j)})}}\Big((T_{\gs^{(j)}}(\frac 12)-a_q)^2+ (a_q+\frac 1{t})^2\Big)\\
&=\sum_{j=1}^k\frac 1{2^{\ell(\gs^{(j)})}} \frac 1{9^{\ell(\gs^{(j)})}}V+\sum_{j=1}^k\frac 1{2^{\ell(\gs^{(j)})}}\Big(2a_q^2-2a_q(T_{\gs^{(j)}} (\frac 12)-\frac 1 t) +(T_{\gs^{(j)}} (\frac 12))^2+\frac 1{t^2}\Big)\\
&=\sum_{j=1}^k\frac 1{2^{\ell(\gs^{(j)})}} \frac 1{9^{\ell(\gs^{(j)})}}V+\sum_{j=1}^k\frac 1{2^{\ell(\gs^{(j)})}}\frac 12\Big(\Big(2a_q- (T_{\gs^{(j)}} (\frac 12)-\frac 1 t)\Big)^2 +\Big(T_{\gs^{(j)}} (\frac 12) +\frac 1{t}\Big)^2\Big). 
\end{align*}
Notice that the above expression is minimum if both the expressions 
\[\sum_{j=1}^k\frac 1{2^{\ell(\gs^{(j)})}}\frac 12\Big(2a_q- (T_{\gs^{(j)}} (\frac 12)-\frac 1 t)\Big)^2\te{ and } \sum_{j=1}^k\frac 1{2^{\ell(\gs^{(j)})}}\frac 12  \Big(T_{\gs^{(j)}} (\frac 12) +\frac 1{t}\Big)^2\]
are minimum. Since $1\leq t\leq n$, both the expressions are minimum if $t=n$. Once $t= n$, the first expression can further be minimized if 
\[\sum_{j=1}^k\frac 1{2^{\ell(\gs^{(j)})}} \Big((2a_q+ \frac 1 n)-T_{\gs^{(j)}} (\frac 12)\Big)=0\]
  yielding  \[ 2a_q+\frac 1n =\frac{\sum_{j=1}^k\frac 1{2^{\ell(\gs^{(j)})}}T_{\gs^{(j)}} (\frac 12)}{\sum_{j=1}^k\frac 1{2^{\ell(\gs^{(j)})}}}.\]
Thus, we have \[a_q=\frac 12\Big(\frac{\sum_{j=1}^k\frac 1{2^{\ell(\gs^{(j)})}}T_{\gs^{(j)}} (\frac 12)}{\sum_{j=1}^k\frac 1{2^{\ell(\gs^{(j)})}}}-\frac 1n\Big)\te{ and } b_q=\frac 12\Big(\frac{\sum_{j=1}^k\frac 1{2^{\ell(\gs^{(j)})}}T_{\gs^{(j)}} (\frac 12)}{\sum_{j=1}^k\frac 1{2^{\ell(\gs^{(j)})}}}-\frac 1n\Big)+\frac 1n \]
  implying \[(a_q, b_q)=U_n^{-1}\Big(\frac{\sum_{j=1}^k\frac 1{2^{\ell(\gs^{(j)})}}T_{\gs^{(j)}} (\frac 12)}{\sum_{j=1}^k\frac 1{2^{\ell(\gs^{(j)})}}}\Big)=U_n^{-1}(\pi(E(\tbf X : \tbf X\in M((a_q, b_q)|\ga_n)))).\]
Since $(a_q, b_q)\in \ga_n$ is chosen arbitrarily, the proof of the lemma is complete.  
\end{proof}

 \begin{remark} \label{remM1} 
By \eqref{eq0000} and \eqref{eq0001}, and Lemma~\ref{lemma0}, we can conclude that all the elements in an optimal set of $n$-points must lie on $S_n$ between the two elements $U_n^{-1}(0)$ and $U_n^{-1}(1)$, i.e., between the two elements $(-\frac{1}{2 n},\frac{1}{2 n})$ and $(\frac{n-1}{2 n},\frac{n+1}{2 n})$. If this fact is not true, then the constrained quantization error can be strictly reduced by moving the elements in the optimal set between the elements $(-\frac{1}{2 n},\frac{1}{2 n})$ and $(\frac{n-1}{2 n},\frac{n+1}{2 n})$ on $S_n$, in other words, the $x$-coordinates of all the elements in an optimal set of $n$-points must lie between the two numbers $-\frac{1}{2 n}$ and $\frac{n-1}{2 n}$.
 \end{remark}

\begin{lemma} \label{carl}  We have 
 $1+5+13+17+37 +41 +49 +53 +\cdots \te{ up to } 2^k\te{-terms}=6^k$, and 
 $1^2+5^2+13^2+17^2+37^2+41^2+49^2+53^2+\cdots \te{ up to } 2^k\te{-terms}=2^{k-1}(9^k3-1)$. 
\end{lemma} 

\begin{proof}
For $k\in \mathbb{N}\cup\{0\}$ let us define sets $\mathfrak{C}_0={1}$ and $\mathfrak{C}_k=\mathfrak{C}_{k-1}\cup (\mathfrak{C}_{k-1}+ (3^{k-1}4))$. Then, notice that $\mathfrak{C}_k=\set{1,5,13,17,\ldots \te{up to $2^k$-terms}}$. Let us next define the moment function for the required sum as follows: 
$\mathcal{M}_{m}(k)=\sum_{x\in \mathfrak{C}_{k}}x^m.$
Then, for $m=0$, we get
$\mathcal{M}_0(k)=\sum_{x\in \mathfrak{C}_k}x^0=2^k.$
For $m=1$, we have
\begin{align*}
&\mathcal{M}_1(k)\\
&=\sum_{x\in \mathfrak{C}_k}x=\sum_{x\in \mathfrak{C}_{k-1}}(x+(x+3^{k-1}4))
= \sum_{x\in \mathfrak{C}_{k-1}}x +\sum_{x\in \mathfrak{C}_{k-1}}x +\sum_{x\in \mathfrak{C}_{k-1}} 3^{k-1}4
=2\mathcal{M}_1 (k-1)+2^{k-1}3^{k-1}4\\
&=2^2\mathcal{M}_1(k-2)+2^{k-1}3^{k-2}4+2^{k-1}3^{k-1}4\\
& \quad \vdots\\
&=2^k\mathcal{M}_1(0)+2^{k-1}4+2^{k-1}3\cdot 4+2^{k-1}3^24+\cdots+2^{k-1}3^{k-2}4+2^{k-1}3^{k-1}4\\
& =2^{k}\cdot 1+2^{k-1}4(1+3+3^2+\cdots+3^{k-1}) 
= 2^k+2^{k+1}\left(\frac{3^{k}-1}{2}\right)= 6^k.
\end{align*}
For $m=2$, we have
\begin{align*}
&\mathcal{M}_2(k)\\
&=\sum_{x\in \mathfrak{C}_k} x^2
=\sum_{x\in \mathfrak{C}_{k-1}} (x^2+(x+3^{k-1}4)^2)= \sum_{x\in \mathfrak{C}_{k-1}}x^2+\sum_{x\in \mathfrak{C}_{k-1}} (x^2+2x\cdot 3^{k-1}4+9^{k-1}16)\\
&=2\sum_{x\in \mathfrak{C}_{k-1}} x^2 +3^{k-1}8\sum_{x\in \mathfrak{C}_{k-1}} x+\sum_{x\in \mathfrak{C}_{k-1}} 9^{k-1}16\\
&= 2\mathcal{M}_2(k-1)+3^{k-1}8\mathcal{M}_1(k-1)+2^{k-1}9^{k-1}16 
=2\mathcal{M}_2(k-1)+3^{k-1}6^{k-1}8+18^{k-1}16 \\
&= 2\mathcal{M}_2(k-1)+18^{k-1}24= 2^2\mathcal{M}_2(k-2)+2\cdot 18^{k-2}24+18^{k-1}24 \\
&= 2^3\mathcal{M}_2(k-3)+2^2 18^{k-3}24+18^{k-1}24\\
&   \quad \vdots\\
&= 2^k\mathcal{M}_2(0)+2^{k-1}24+2^{k-2}18\cdot 24+\cdots +2^2 18^{k-3}24+18^{k-1}24\\
&=2^k\cdot 1+2^{k-1}24(1+9+9^2+\cdots+9^{k-1})=2^k+2^{k-1}24\left(\frac{9^k-1}{8}\right)=2^{k-1}(9^k3-1).
\end{align*}
Therefore, $1+5+13+17+37 +41 +49 +53 +\cdots \te{ up to } 2^k\te{-terms}=6^k$, and 
 $1^2+5^2+13^2+17^2+37^2+41^2+49^2+53^2+\cdots \te{ up to } 2^k\te{-terms}=2^{k-1}(3\cdot 9^k-1)$.
 Thus, the proof of the lemma is complete. 
\end{proof} 
 
\begin{defi1}  \label{defi1} 
For $n\in \D N$ with $n\geq 2$, let $\ell(n)$ be the unique natural number with $2^{\ell(n)} \leq n<2^{\ell(n)+1}$. Let $U_n$ be the mappings given by \eqref{eq0001}. For $I\sci \set{1, 2}^{\ell(n)}$ with card$(I)=n-2^{\ell(n)}$ let $\ga_n(I)\ci S_n$ be the set such that
\[\ga_n(I)=\set{U_n^{-1}(a(\gs)) : \gs \in \set{1,2}^{\ell(n)} \setminus I} \uu \set{U_n^{-1}(a(\gs 1)) : \gs \in I} \uu \set {U_n^{-1}(a(\gs 2)) : \gs \in I}.\]
\end{defi1}

\begin{prop} \label{prop0}
Let $\ga_n(I)$ be the set given by Definition \ref{defi1}. Then, the number of such sets is  ${}^{2^{\ell(n)}}C_{n-2^{\ell(n)}}$, and the corresponding distortion error is given by
\begin{equation*}\label{eq11}
V(P; \ga_n(I))=\int\mathop{\min}\limits_{a\in\ga_n(I)} \rho(x, a)\, dP=\frac 1 {18^{\ell(n)}}V\Big(2^{\ell(n)+1}-n+\frac 1 9(n-2^{\ell(n)})\Big)+A,
\end{equation*}
where $V$ is the variance as given by Lemma~\ref{lemma2}, and 
\[A =\sum_{\gs \in \set{1, 2}^{\ell(n)}\setminus I}\frac 1 {2^{\ell(n)}} \rho(a(\gs), U_n^{-1}(a(\gs))) +\sum_{\gs \in I}\frac 1{2^{\ell(n)+1}} \Big( \rho(a(\gs1), U_n^{-1}(a(\gs1)))+\rho(a(\gs2), U_n^{-1}(a(\gs2)))\Big).\]
\end{prop}
\begin{proof}
If $2^{\ell(n)} \leq n<2^{\ell(n)+1}$, then the subset $I$ can be chosen in ${}^{2^{\ell(n)}}C_{n-2^{\ell(n)}}$ different ways, and so, the number of such sets is given by ${}^{2^{\ell(n)}}C_{n-2^{\ell(n)}}$, and the corresponding distortion error is obtained as
\begin{align*}
&V(P; \ga_n(I))=\int \min_{a\in \ga_n(I)}\rho(x, a)\, dP\\
&=\sum_{\gs \in \set{1, 2}^{\ell(n)}\setminus I}\int_{J_\gs}\rho(x, U_n^{-1}(a(\gs)))\,dP\\
&\qquad +\sum_{\gs \in I} \Big(\int_{J_{\gs1}}\rho(x,U_n^{-1}(a(\gs1)))\,dP +\int_{J_{\gs2}}\rho(x, U_n^{-1}(a(\gs2)))\,dP\Big)\\
&=\sum_{\gs \in \set{1, 2}^{\ell(n)}\setminus I}\frac 1 {2^{\ell(n)}} \int\rho(T_\gs(x), U_n^{-1}(a(\gs)))\,dP\\
&\qquad +\sum_{\gs \in I}\frac 1{2^{\ell(n)+1}} \Big(\int\rho(T_{\gs1}(x), U_n^{-1}(a(\gs1)))\,dP+\int\rho(T_{\gs2}(x), U_n^{-1}(a(\gs2))\,dP\Big)\\
&=\sum_{\gs \in \set{1, 2}^{\ell(n)}\setminus I}\frac 1 {2^{\ell(n)}}\Big(\frac 1 {9^{\ell(n)}}V +\rho(a(\gs), U_n^{-1}(a(\gs)))\Big)\\
&\qquad +\sum_{\gs \in I}\frac 1{2^{\ell(n)+1}} \Big(\frac 2{9^{\ell(n)+1}} V+ \rho(a(\gs1), U_n^{-1}(a(\gs1)))+\rho(a(\gs2), U_n^{-1}(a(\gs2)))\Big)\\
 &=\frac 1 {18^{\ell(n)}}V\Big(2^{\ell(n)+1}-n+\frac 1 9(n-2^{\ell(n)})\Big)+A,
\end{align*}
where
\[
A =\sum_{\gs \in \set{1, 2}^{\ell(n)}\setminus I}\frac 1 {2^{\ell(n)}} \rho(a(\gs), U_n^{-1}(a(\gs))) +\sum_{\gs \in I}\frac 1{2^{\ell(n)+1}} \Big( \rho(a(\gs1), U_n^{-1}(a(\gs1)))+\rho(a(\gs2), U_n^{-1}(a(\gs2)))\Big).
\]
Thus, the proof of the proposition is complete. 
\end{proof}

The following corollary is a consequence of Proposition~\ref{prop0}. 
\begin{corollary} \label{cor001} 
Let $A$ be the expression given in Proposition~\ref{prop0}. Then, if $n$ is of the form $n=2^{\ell(n)}$ for some positive integer $\ell(n)\in \D N$, we have 
\begin{equation*} 
A=\frac{2^{\ell(n)}+1}{2\cdot 4^{\ell(n)}}+\frac{3\cdot 9^{\ell(n)}-1}{16\cdot 9^{\ell(n)}}.
\end{equation*}
\end{corollary} 

\begin{proof} Let $n\in \D N$ be such that $n$ is of the form $n=2^{\ell(n)}$ for some positive integer $\ell(n)\in \D N$. Notice that for $\gs \in \set{1, 2}^{\ell(n)}$, by \eqref{Me00} we have 
\[\rho(a(\gs), U_n^{-1}(a(\gs)))=\rho\Big(a(\gs), \Big(\frac 12(a(\gs)-\frac 1n), \frac 12(a(\gs)-\frac 1n)+\frac 1 n\Big)\Big)= \frac 12(a(\gs)+\frac 1{n})^2.\]
Thus, using Lemma~\ref{carl}, we have   
\begin{align*}
A&=\sum_{\gs \in \set{1, 2}^{\ell(n)}}\frac 1 {2^{\ell(n)}} \frac 12(a(\gs)+\frac 1{2^{\ell(n)}})^2 =\sum_{\gs \in \set{1, 2}^{\ell(n)}}\frac 1 {2^{\ell(n)}}\frac 12\Big((a(\gs))^2+2 a(\gs)\cdot\frac 1{2^{{\ell(n)}}}+\frac 1{4^{\ell(n)}}\Big) \\ 
&=\frac 1 {2^{{\ell(n)}+1}}\cdot \frac 1{(2\cdot 3^{\ell(n)})^2}\Big(1^2+5^2+13^2+17^2+37^2+41^2+49^2+53^2+\cdots \te{ up to } 2^{\ell(n)}\te{-terms}\Big)\\
&\qquad \qquad +\frac 1{4^{\ell(n)}}\cdot \frac 1{2\cdot 3^{\ell(n)}}\Big(1+5+13+17+37 +41 +49 +53 +\cdots \te{ up to } 2^{\ell(n)}\te{-terms}\Big)+\frac 12 \cdot \frac 1{4^{\ell(n)}}\\
&=\frac{1}{2\cdot 4^{\ell(n)}}+\frac{6^{\ell(n)}}{\left(2\cdot 3^{\ell(n)}\right) 4^{\ell(n)}}+\frac{2^{{\ell(n)}-1} \left(3\cdot 9^{\ell(n)}-1\right)}{2^{{\ell(n)}+1} \left(2\cdot 3^{\ell(n)}\right)^2}\\
&=\frac{2^{\ell(n)}+1}{2\cdot 4^{\ell(n)}}+\frac{3\cdot 9^{\ell(n)}-1}{16\cdot 9^{\ell(n)}}.
\end{align*}
Thus, the proof of the corollary is yielded. 
\end{proof}

In the next sections, we give the main results of the paper. 

 \begin{figure}
\vspace{-0.25 in}
\centerline{\includegraphics[width=6 in, height=4in]{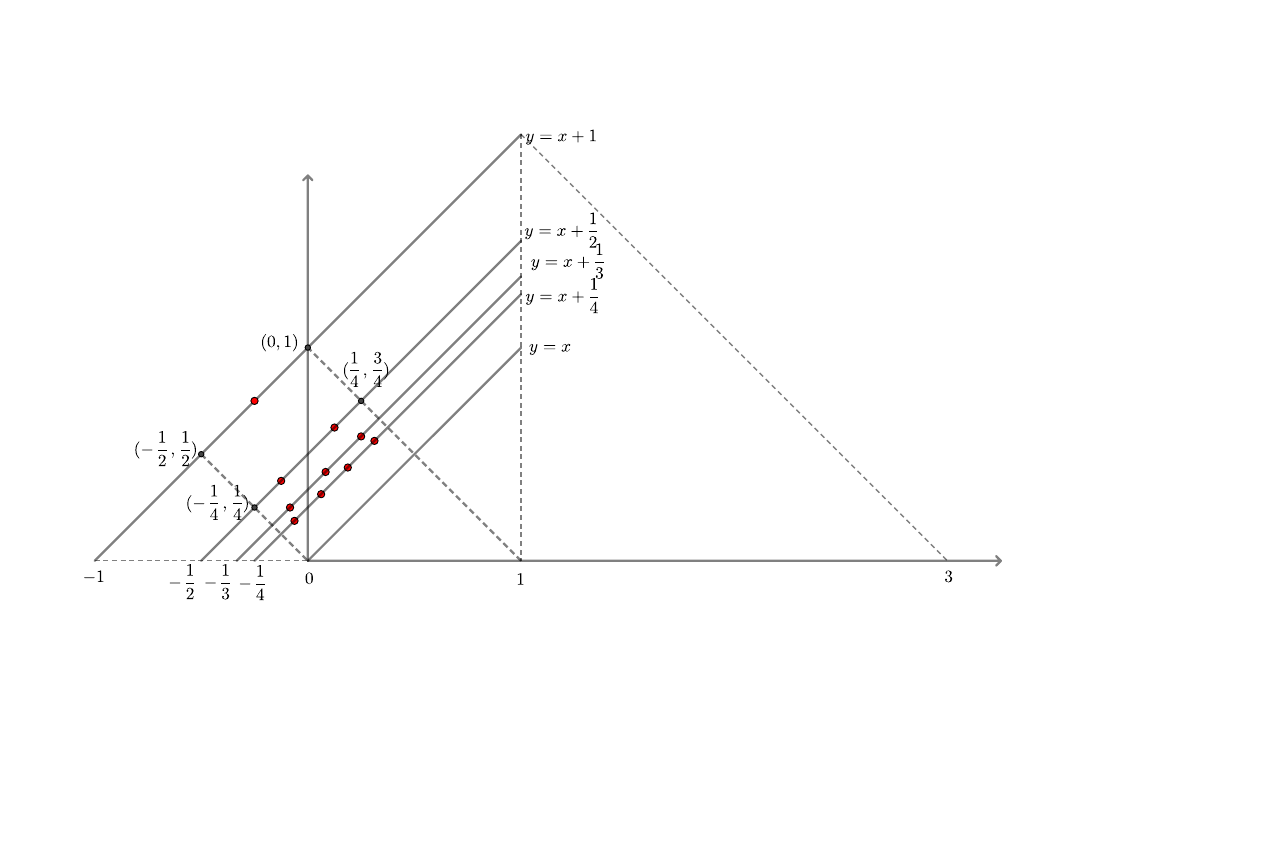}}
\vspace{-1.2 in}
\caption{Points in the optimal sets of $n$-points for $1\leq n\leq 4$.} \label{Fig}
\end{figure}

\section{Main Results} \label{sec1}

In this section, Theorem~\ref{Megha0}, Theorem~\ref{theo2}, and Theorem~\ref{theo3} contain all the main results of the paper.

\begin{prop}
A constrained optimal set of one-point is $\set{(-\frac{1}{4},\frac{3}{4})}$ with constrained quantization error $V_1=\frac{5}{4}$.
\end{prop}
\begin{proof}
Let $\ga:=\set{(a, b)}$ be a constrained optimal set of one-point. Since $\ga \ci S_1$, we have $b=a+1$. Now, the distortion error for $P$ with respect to the set $\ga$ is give by
\[V(P; \ga)=\int \rho(x, (a, a+1)) dP=2 a^2+a+\frac{11}{8},\]
the minimum value of which is $\frac{5}{4}$ and it occurs when $a=-\frac 14$. Thus, a constrained optimal set of one-point is $\set{(-\frac{1}{4},\frac{3}{4})}$ with constrained quantization error $V_1=\frac 54$, which is the proposition. 
\end{proof}
The following proposition is known. 

\begin{prop} (see \cite{GL2}) \label{Me420} 
For $n\geq 2$, let $\ga_n(I)$ be the set given by Definition~\ref{defi1}, and for each $j\in \mathbb{N}$, let $U_j$ be the bijective mapping as defined by \eqref{eq0001}.
Then, the set 
\[U_n(\ga_n(I))=\set{a(\gs) : \gs \in \set{1,2}^{\ell(n)} \setminus I} \uu \set{a(\gs 1) : \gs \in I} \uu \set {a(\gs 2) : \gs \in I}\]
forms an unconstrained optimal set of $n$-means for the Cantor distribution $P$ with the $n$th unconstrained quantization error 
\[V(P; U_n(\ga_n(I)))=\frac 1 {18^{\ell(n)}}V\Big(2^{\ell(n)+1}-n+\frac 1 9(n-2^{\ell(n)})\Big).\]
\end{prop} 
 
 \begin{prop} \label{Me421} 
The bijective mappings $U_n$ preserves the Voronoi regions with respect to the probability measure $P$, i.e., for any discrete $\gb\sci \D R$, and $a\in \gb$, we have 
\begin{equation*} \label{Me45} P(M(a|\gb))=P(M(U_n^{-1}(a)|U_n^{-1}(\gb))).
\end{equation*} 
 \end{prop} 
 \begin{proof}
Since $\gb\sci \D R$, for any $a\in \gb$, we can write $M(a|\gb)=[c, d]$ for some $c, d\in \D R$ with $c<d$.  
 Notice that the bijective mapping $U_n$ preserves the order, i.e., for any $(e, e+\frac 1 n), \, (f, f+\frac 1 n) \in S_n$ if $e<f$, then $U_n(e, e+\frac 1 n)<U_n(f, f+\frac 1 n)$. 
 Moreover, for any $(e, e+\frac 1n)\in S_n$, $U_n(e, e+\frac 1n)$ represents the point on $J$ where the perpendicular on $S_n$ at $(e, e+\frac 1n)$ intersects $J$. Hence, we can say that the boundary of the Voronoi region $M(U_n^{-1}(a)|U_n^{-1}(\gb))$ intersects $S_n$ at the points given by $U_n^{-1}(c)$ and $U_n^{-1}(d)$, i.e., $M(U_n^{-1}(a)|U_n^{-1}(\gb))$ contains the closed interval $[c, d]$ as a subset, i.e.,  
  \[M(a|\gb)\sci M(U_n^{-1}(a)|U_n^{-1}(\gb)).\]
 Since $P(M(U_n^{-1}(a)|U_n^{-1}(\gb)) \setminus M(a|\gb))=0$, we have $P(M(a|\gb))=P(M(U_n^{-1}(a)|U_n^{-1}(\gb)))$. Thus, the proof of the proposition is complete. 
 \end{proof} 
The following theorem gives the optimal sets of $n$-points for all positive integers $n\geq 2$ for the Cantor distribution $P$ with respect to the family of constraints $\set{S_j : j\in \D N}$. 

\begin{theorem} \label{Megha0}
For $n\geq 2$, let $\ga_n(I)$ be the set given by Definition~\ref{defi1}. Then, $\ga_n(I)$ forms a constrained optimal set of $n$-points for $P$ with $n$th constrained quantization error 
\[V_n=V(P; \ga_n(I))=V(P; U_n(\ga_n(I)))+A.\]
\end{theorem} 

\begin{proof}
To prove that $\ga_n(I)$ forms a constrained optimal set of $n$-points for $P$, it is enough to prove the fact that $\ga_n(I)$ forms a constrained optimal set of $n$-points for $P$ if and only if $U_n(\ga_n(I))$ forms an unconstrained optimal set of $n$-means for $P$. 
The fact is clearly true by Proposition~\ref{Me420} and Proposition~\ref{Me421}. Hence, $\ga_n(I)$ forms a constrained optimal set of $n$-points for $P$ (see Figure~\ref{Fig}).  Then, by Proposition \ref{prop0} and Proposition~\ref{Me420}, we have the $n$th constrained quantization error as 
\[V_n=V(P; \ga_n(I))=V(P; U_n(\ga_n(I)))+A.\]
Thus, the proof of the theorem is complete. 
\end{proof} 
 
We need the following proposition, which is a special case of Theorem~\ref{Megha0}, to prove Theorem~\ref{theo2} and Theorem~\ref{theo3}. 

\begin{prop} \label{Mega422} 
Let $n\in \D N$ be such that $n=2^{\ell(n)}$ for some positive integer $\ell(n)$. Then,
the set \[\ga_n(I)=\set{U_n^{-1}(a(\gs)) : \gs \in \set{1,2}^{\ell(n)}}\] forms a constrained optimal set of $2^{\ell(n)}$-points with constrained quantization error
\[V_{2^{\ell(n)}}(P)=\frac{1}{16} \left(2^{3-2 \ell(n)}+2^{3-\ell(n)}+9^{-\ell(n)}+3\right).\]
\end{prop}
\begin{proof}
Let $n=2^{\ell(n)}$ for some positive integer $\ell(n)$. By Theorem~\ref{Megha0}, it follows that the set $\set{U_n^{-1}(a(\gs)) : \gs \in \set{1,2}^{\ell(n)}}$ forms a constrained optimal set of $n$-points. 
By Proposition~\ref{Me420} and Theorem~\ref{Megha0}, and Corollary~\ref{cor001}, it follows that the $n$th constrained quantization error is 
\[V_{2^{\ell(n)}}(P)=\frac{V}{9^{\ell(n)}}+\frac{2^{\ell(n)}+1}{2\cdot 4^{\ell(n)}}+\frac{3\cdot 9^{\ell(n)}-1}{16\cdot 9^{\ell(n)}},\]
which yields \[V_{2^{\ell(n)}}(P)=\frac{1}{16} \left(2^{3-2 \ell(n)}+2^{3-\ell(n)}+9^{-\ell(n)}+3\right).\]
Thus, the proof of the proposition is complete. 
\end{proof}

\begin{theorem}\label{theo2} 
The constrained quantization dimension $D(P)$ of the probability measure $P$ exists, and $D(P)=2$. 
\end{theorem}

\begin{proof}
For $n\in \D N$ with $n\geq 2$, let $\ell(n)$ be the unique natural number such that $2^{\ell(n)}\leq n<2^{\ell(n)+1}$. Then,
$V_{2^{\ell(n)+1}}\leq V_n\leq V_{2^{\ell(n)}}$. By Proposition~\ref{Mega422}, we see that $V_{2^{\ell(n)+1}}\to \frac{3}{16}$ and  $V_{2^{\ell(n)}}\to \frac{3}{16} $ as $n\to \infty$, and so $V_n\to \frac{3}{16}$ as $n\to \infty$,
i.e., $V_\infty=\frac{3}{16}$.
We can take $n$ large enough so that $(V_{2^{\ell(n)}}-V_\infty)<1$. Then,
\[0<-\log (V_{2^{\ell(n)}}-V_{\infty})\leq -\log (V_n-V_\infty)\leq -\log (V_{2^{\ell(n)+1}}-V_\infty)\]
yielding
\[\frac{2\ell(n) \log 2}{-\log (V_{2^{\ell(n)+1}}-V_\infty)}\leq \frac{2\log n}{-\log(V_n-V_\infty)}\leq \frac{2(\ell(n)+1) \log 2}{-\log (V_{2^{\ell(n)}}-V_\infty)}.\]
Notice that
\begin{align*}
&\lim_{n\to \infty} \frac{2\ell(n) \log 2}{-\log (V_{2^{\ell(n)+1}}-V_\infty)}=\lim_{n\to \infty} \frac{2\ell(n) \log 2}{-\log(\frac{1}{16} \left(2^{2-\ell(n)}+2^{3-2 (\ell(n)+1)}+9^{-\ell(n)-1}+3\right)-\frac{3}{16})} 
\end{align*}
implying \[\lim_{n\to \infty} \frac{2\ell(n) \log 2}{-\log (V_{2^{\ell(n)+1}}-V_\infty)}=2. 
\te{ Similarly, } 
 \lim_{n\to \infty} \frac{2(\ell(n)+1) \log 2}{-\log (V_{2^{\ell(n)}}-V_\infty)}=2.\]
Hence, $\lim_{n\to \infty}  \frac{2\log n}{-\log(V_n-V_\infty)}=2$, i.e., the constrained quantization dimension $D(P)$ of the probability measure $P$ exists and $D(P)=2$.
Thus, the proof of
the theorem is complete.
\end{proof}

\begin{theorem} \label{theo3} 
The $D(P)$-dimensional constrained quantization coefficient for $P$ is infinity.
\end{theorem}
\begin{proof}
For $n\in \D N$ with $n\geq 2$, let $\ell(n)$ be the unique natural number such that $2^{\ell(n)}\leq n<2^{\ell(n)+1}$. Then,
$V_{2(\ell(n)+1)}\leq V_n\leq V_{2{\ell(n)}}$, and $V_\infty=\lim_{n\to\infty} V_n=\frac 3{16}$. Since
\begin{align*}
&\lim_{n\to \infty} n^2 (V_n-V_\infty)\geq \lim_{n\to \infty} (2^{\ell(n)})^2 (V_{2^{\ell(n)+1}}-V_\infty)\\
&=\lim_{n\to\infty}(2\ell(n))^2\Big(\frac{1}{16} \left(2^{2-\ell(n)}+2^{3-2 (\ell(n)+1)}+9^{-\ell(n)-1}+3\right)-\frac{3}{16}\Big)=\infty, \te{ and } \\
&\lim_{n\to \infty} n^2 (V_n-V_\infty)\leq \lim_{n\to \infty} (2^{\ell(n)+1})^2 (V_{2^{\ell(n)}}-V_\infty)\\
&=\lim_{n\to\infty}(2^{\ell(n)+1})^2\Big(\frac{1}{16} \left(2^{3-2 \ell(n)}+2^{3-\ell(n)}+9^{-\ell(n)}+3\right)-\frac{3}{16}\Big) =\infty,
\end{align*}
by squeeze theorem, we have  $\lim_{n\to \infty} n^2 (V_n-V_\infty)=\infty$, which is the theorem. 
\end{proof}


\begin{thebibliography}{9999}
\bibitem[BW]{BW} J.A. Bucklew and G.L. Wise, \emph{Multidimensional asymptotic quantization theory with $r$th power distortion measures}, IEEE Transactions on Information Theory, 1982, Vol. 28 Issue 2, 239-247.
    
    \bibitem[BCDRV]{BCDRV} P. Biteng, M. Caguiat, D. Deb,  M.K. Roychowdhury, and B. Villanueva, \emph{Constrained quantization for a uniform distribution}, to appear, Real Analysis Exchange. 


  \bibitem[DFG]{DFG} Q. Du, V. Faber and M. Gunzburger, \emph{Centroidal Voronoi Tessellations: Applications and Algorithms}, SIAM Review, Vol. 41, No. 4 (1999), pp. 637-676.

\bibitem[DR]{DR} C.P. Dettmann and M.K. Roychowdhury, \emph{Quantization for uniform distributions on equilateral triangles}, Real Analysis Exchange, Vol. 42(1), 2017, pp. 149-166.

\bibitem[GG]{GG} A. Gersho and R.M. Gray, \emph{Vector quantization and signal compression}, Kluwer Academy publishers: Boston, 1992.


\bibitem[GL]{GL} S. Graf and H. Luschgy, \emph{Foundations of quantization for probability distributions}, Lecture Notes in Mathematics 1730, Springer, Berlin, 2000.


\bibitem[GL1]{GL1} A. Gy\"orgy and T. Linder, \emph{On the structure of optimal entropy-constrained scalar quantizers},  IEEE transactions on information theory, vol. 48, no. 2, February 2002.

\bibitem [GL2]{GL2} S. Graf and H. Luschgy, \emph{The Quantization of the Cantor Distribution}, Math. Nachr., 183 (1997), 113-133.


\bibitem[GL3]{GL3} S. Graf and H. Luschgy, \emph{Quantization for probability measures with respect to the geometric mean error}, Math. Proc. Camb. Phil. Soc. (2004), 136, 687-717. 



  \bibitem[GN]{GN} R.M. Gray and D.L. Neuhoff, \emph{Quantization}, IEEE Transactions on Information Theory, Vol. 44, No. 6, October 1998, 2325-2383.


 \bibitem[H]{H} J. Hutchinson, \emph{Fractals and self-similarity}, Indiana Univ. J., 30 (1981), 713-747.


\bibitem[KNZ]{KNZ}  M. Kesseb\"ohmer, A. Niemann and S. Zhu, \emph{Quantization dimensions of compactly supported probability measures via R\'enyi dimensions}, Trans. Amer. Math. Soc. (2023). 

\bibitem[P]{P} D. Pollard, \emph{Quantization and the Method of $k$-Means}, IEEE Transactions on Information Theory, 28 (1982), 199-205.
  \bibitem[P1]{P1} K. P\"otzelberger, \emph{The quantization dimension of distributions}, Math. Proc. Cambridge Philos. Soc., 131 (2001), 507-519.

\bibitem[PR1]{PR1} M. Pandey and M.K. Roychowdhury, \emph{Constrained quantization for probability distributions}, arXiv:2305.11110 [math.PR].
\bibitem[PR2]{PR2} M. Pandey and M.K. Roychowdhury, \emph{Constrained quantization for the Cantor distribution}, arXiv:2306.16653 [math.DS].
\bibitem[PR3]{PR3} M. Pandey and M.K. Roychowdhury, \emph{Constrained quantization for a uniform distribution with respect to a family of constraints}, arXiv:2309:11498 [math.PR].
\bibitem[PR4]{PR4} M. Pandey and M.K. Roychowdhury, \emph{Conditional constrained and unconstrained quantization for probability distributions}, arXiv:2312:02965 [math.PR].


\bibitem[RR]{RR} J. Rosenblatt and M.K. Roychowdhury, \emph{Uniform distributions on curves and quantization}, Commun. Korean Math. Soc. 38 (2023), No. 2, pp. 431-450.


\bibitem[R1]{R1} M.K. Roychowdhury, \emph{Quantization and centroidal Voronoi tessellations for probability measures on dyadic Cantor sets}, Journal of Fractal Geometry, 4 (2017), 127-146.

\bibitem [R2]{R2} M.K. Roychowdhury, \emph{Least upper bound of the exact formula for optimal quantization of some uniform Cantor distributions}, Discrete and Continuous Dynamical Systems- Series A, Volume 38, Number 9, September 2018, pp. 4555-4570.
\bibitem[R3]{R3} M.K. Roychowdhury, \emph{Optimal quantization for the Cantor distribution generated by infinite similitudes}, Israel Journal of Mathematics 231 (2019), 437-466.


\bibitem[Z1]{Z1} P.L. Zador, \emph{Asymptotic Quantization Error of Continuous Signals and the Quantization Dimension}, IEEE Transactions on Information Theory, 28 (1982), 139-149.

\bibitem[Z2]{Z2} R. Zam, \emph{Lattice Coding for Signals and Networks: A Structured Coding Approach to Quantization, Modulation, and Multiuser Information Theory}, Cambridge University Press, 2014.

\end{thebibliography}
\end{document}